\documentclass{amsart}

\usepackage{amsfonts,amssymb,amscd,amsmath,enumerate,verbatim,calc,cite}

\newcommand{\gin}{\operatorname{gin}}

\newcommand{\ini}{\operatorname{In}}

\newcommand{\tra}{\operatorname{Tr}}

\newcommand{\tr}{\operatorname{Tr}}
\theoremstyle{plain}
\newtheorem{theorem}{Theorem}

\newtheorem{lemma}[theorem]{Lemma}

\newtheorem{conjecture}[theorem]{Conjecture}
\newtheorem{proposition}[theorem]{Proposition}

\theoremstyle{definition}

\theoremstyle{remark}
\newtheorem{remark}[theorem]{Remark}

\newcounter{hours}\newcounter{minutes}
\newcommand{\printtime}{%
        \setcounter{hours}{\time/60}%
        \setcounter{minutes}{\time-\value{hours}*60}%
        \thehours\,h\ \theminutes\,min}

\begin{document}

\title[ Gin of modular invariants]
{ generic initial ideals of modular polynomial invariants}
\date{\today \printtime}

\author{ Bek\.ir Dan{\i}\c{s}}
\author{M\"uf\.it Sezer}
\email{bekir.danis@bilkent.edu.tr}

\email{sezer@fen.bilkent.edu.tr}
\address { Department of Mathematics, Bilkent University,
 Ankara 06800 Turkey}
\thanks{Second author is supported by a grant from T\"ubitak:115F186}


\begin{abstract}
We study  the generic initial ideals ($ \gin $) of certain ideals
that arise in modular  invariant theory.  For all cases an
explicit generating set is known we compute the generic initial
ideal of the Hilbert ideal of a cyclic group of prime order  for
all monomial orders. We also consider the Klein four group and
note that its Hilbert ideals are Borel fixed with certain
orderings of the variables. In all situations we consider, it is
possible to select a monomial order such that  the gin of the
Hilbert ideal is equal to its initial ideal. As an incidental
result we show that $\gin$ respects a permutation of the variables
in the monomial order.
\end{abstract}

\maketitle

\section{Introduction}
For a homogeneous ideal $I$ in a polynomial ring, its generic
initial ideal $\gin_{>}(I)$ with respect to a term order $>$  is
the ideal of initial monomials after a generic change of
coordinates. The generic initial ideal  measures how close to a
$>$-segment ideal $I$ can be made by  a linear and invertible
substitution. It encodes much information on the combinatorial,
geometrical and homological  properties of $I$ and the associated
variety  and plays an important role in the computational aspects
of commutative algebra and algebraic geometry. For instance
generic initial ideals were used in Hartshorne's proof of the
connectedness of Hilbert schemes. Describing $ \gin_{>}(I)$  is a
very difficult task in general and despite their significance,
there are relatively few classes of ideals for which generic
initial ideals are explicitly computed. We refer the reader  to
\cite{MR1648665}  for a survey of results on this matter.

In this paper we study generic initial ideals that arise in
invariant theory. We consider  a finite dimensional module $V$ of
a group $G$ over an infinite field $F$. There is an induced action
on the symmetric algebra $F[V]:=S(V^*)$ on $V^*$. This is a
polynomial algebra $F[x_1, \dots , x_n]$, where $x_1, \dots , x_n$
is a basis for   $V^*$. A classical object is the ring of
invariants $F[V]^G:=\{ f\in F[V] \mid \sigma (f)=f  \; \text { for
all } \sigma \in G\}$ which is a graded subalgebra of $F[V]$.  The
ideal in $F[V]$ generated by  homogeneous invariants of positive
degree  is the Hilbert ideal of $V$ and we denote it by $H(V)$.
The Hilbert ideal and its quotient in $F[V]$ often contain
information about the invariant ring itself.  It plays an
important role in finding generators of  $F[V]^G$ or obtaining
degree bounds for them.  When the characteristic of the field
divides the order of the   group, i.e., $V$ is a modular module,
the   invariant ring is  more complicated and difficult to obtain.
Invariants are not known in general  even in the simplest modular
situation when $G$ is a cyclic  group of prime order. For this
group we consider the cases where an explicit generating set is
known for the Hilbert ideal, and we compute the  generic initial
ideals of these Hilbert ideals for all orders. It turns out that,
with the upper triangular ordering of the variables,  $\gin$ is
equal to the initial ideal of the Hilbert ideal in  these cases. A
natural question one encounters in these computations is how
$\gin$ changes when the variables in the monomial order are
permuted. Since we did not find a source in the literature that
addresses this question, we include a compact proof of the fact
that $\gin$ is also permuted in the same way in Section
\ref{permutation}.
 Case by case analysis and
computations of the $\gin$ of the modular Hilbert ideals of a
cyclic group of prime order are done in Section \ref{hilbert}. In
the final section  we note that any Hilbert ideal of the Klein
four group in characteristic two is Borel fixed with the right
choice of the monomial order and therefore gin is equal to the
Hilbert ideal itself. We feel that these findings support
Conjecture \ref{conj} which states that  for a given module over
characteristic $p$ of a  $p$-group there is always a choice of
basis for the module  and a monomial order such that gin of the
Hilbert ideal is equal to its initial ideal.

We refer the reader to \cite[\S 4]{MR2724673} and to
\cite{MR2759466, MR3445218} for more background on generic initial
ideals and invariant theory, respectively.

\section{Permutation of variables and gin}
\label{permutation}

 In this section we do not consider group actions so  we let  $S$ (instead of $F[V])$ denote the poynomial ring $F[x_1, \dots , x_n]$ in $n$
 variables. Let $I$ be a homogeneous ideal in $S$. We fix a term
 order $>$ on the set of monomials in $S$. The largest monomial that appears in a polynomial $f\in S$ is called the initial monomial of $f$ and is denoted by $\ini_{>} (f)$. We
  denote the ideal in $S$ generated by the initial monomials of elements in $I$ with $\ini_{>} (I)$. Let $\pi$ be a
 permutation of $\{1, 2, \dots  , n \}$. Then $\pi$ induces an
 isomorphism of $S$ via $\pi (x_i)=x_{\pi (i)}$. Note that this isomorphism sends monomials to monomials. Let
 $>_{\pi}$ denote the term order such that

 $$\pi (M_1)>_{\pi}\pi (M_2)  \text { if and only if   } M_1>M_2.$$

 Let $S_d$ denote the $d$-th homogeneous  component of $S$ and we
 consider the $t$-th exterior power $\wedge^tS_d$ of  $S_d$.
Recall that an element $m_1\wedge m_2 \cdots \wedge m_t$, where
$m_i$ is a monomial of degree $d$ with $m_1>m_2> \cdots >m_t$, is
called a standard exterior monomial of $\wedge^tS_d$ with respect
to $>$.
  One can order
standard exterior monomials lexicographically: If $m_1\wedge m_2
\cdots \wedge m_t$ and $w_1\wedge w_2 \cdots \wedge w_t$ are two
standard exterior monomials with respect to $>$, then we set

 $$m_1\wedge m_2
\cdots \wedge m_t> w_1\wedge w_2 \cdots \wedge w_t$$ if $m_i>w_i$
for the smallest index $i$ with $m_i\neq w_i$. We denote the
largest standard exterior monomial in the support of $e\in
\wedge^tS_d$ with $\ini_{>} (e)$. Standard exterior monomials with
respect to $>_{\pi}$ are defined and  ordered similarly.

Let $\textbf{$\alpha$}\in GL_n(F)$. Then
$\textbf{$\alpha$}=(\alpha_{ij})$ induces a degree preserving
isomorphism on $S$ by $x_j\rightarrow
\alpha_{1j}x_1+\alpha_{2j}x_2+ \cdots +\alpha_{nj}x_n$ for $1\le
j\le n$. We consider the polynomial ring in the extended set of
variables $R=F[x_1, \dots , x_n, \alpha_{i,j} \mid 1\le i,j\le
n]$.  We extend $\pi$  to $R$ by letting $\pi
(\alpha_{ij})=\alpha_{\pi (i)j}$.

\begin{lemma}
Let $f\in S$ and $\alpha\in GL_n(F)$. Consider $\alpha (f)$ as a
polynomial in $x_1, \dots , x_n$ with coefficients in
$F[\alpha_{ij}]$. Let $M$ be a monomial in $S$ that appears in
$\alpha (f)$ with coefficient $c\in F[\alpha_{ij}]$. Then the
coefficient of $\pi(M)$ in $\alpha (f)$ is $\pi (c)$.
\end{lemma}
\begin{proof}
Note that $\pi (\alpha (x_j))=\pi (\sum_{1\le i\le
n}\alpha_{ij}x_i)=\sum _{1\le i\le n}\alpha_{\pi (i)j}x_{\pi
(i)}=\alpha (x_j)$. Since both $\pi$ and $\alpha$ are ring
homomorphisms it follows that $\pi (\alpha (f))=\alpha (f)$ for
all $f\in S$. We write $\alpha (f)=\sum c_kM_k$, where $c_k\in
F[\alpha_{ij}]$ and $M_k$ is a monomial in $S$. Then we have
$$\alpha (f)=\pi (\alpha (f))=\sum \pi (c_k)\pi (M_k).$$ Therefore we get $\sum c_kM_k=\sum \pi (c_k)\pi (M_k)$. Since $\pi$ is
a permutation of the monomials in $S$, the assertion of the lemma
follows.
\end{proof}

There is  a Zariski open set $U\subseteq GL_n(F)$ and a monomial
ideal $J$ such that $\ini_{>}(\alpha (I))=J$ for all $\alpha\in
U$. The ideal $J$ is called the generic initial ideal with respect
to $>$ and is denoted $\gin_{>}(I)$, see \cite[4.1.3]{MR2724673}
and \cite[15.18]{MR1322960}.

\begin{theorem}
\label{permutation}
 Let $I$ be a homogeneous ideal. Then we have

$$\gin_{>_{\pi}} (I)= \pi (\gin_{>}(I)).$$
\end{theorem}
\begin{proof}
Consider a homogeneous component $I_d$ of $I$ with a basis $f_1,
\dots ,f_t$. Let $m_{j_1}, \dots , m_{j_t}$   be monomials in $S$
that appear in $\alpha (f_1), \dots ,\alpha (f_t)$ with
coefficients $c_1, \dots , c_t\in F[\alpha_{ij}]$, respectively.
Assume that $c_1m_{j_1}\wedge \cdots \wedge c_tm_{j_t}$
  is a  multiple of the
standard exterior monomial $m_1 \wedge \cdots \wedge m_t$  (with
respect to $>$) in $\wedge^tS_d$. Call this multiple $c\in
F[\alpha_{ij}]$.
 Then
by the previous lemma, $\pi (m_{j_1}), \dots ,\pi(m_{j_t)}$ appear
in $\alpha (f_1), \dots ,\alpha (f_t)$ with coefficients $\pi
(c_1), \dots ,\pi (c_t)\in F[\alpha_{ij}]$, respectively. Since
ranking of the monomials in $>$ is preserved in $>_{\pi}$ after we
apply $\pi$,   it follows that the coefficient of the standard
exterior monomial   $\pi (m_{1}) \wedge \cdots \wedge \pi (m_{t})$
(with respect to $>_{\pi }$)  in $\pi ({c_1})\pi (m_{j_1}) \wedge
\cdots \wedge \pi (c_t)\pi (m_{j_t})$ is $\pi (c)$. Therefore we
have
\begin{alignat*}{1}
\alpha (f_1)\wedge \cdots \wedge \alpha (f_t)=&\sum_{m_1 > \cdots
> m_t}c(m_1, \dots , m_t) (m_1 \wedge \cdots  \wedge m_t) \\
=&\sum_{\pi (m_1) >_{\pi} \cdots >_{\pi} \pi (m_t)}\pi (c(m_1,
\dots , m_t) ) (\pi (m_1) \wedge \cdots \wedge \pi (m_t)).
\end{alignat*}
Since $\pi$ is a permutation of variables in $F[\alpha_{ij}]$,
$c(m_1, \dots , m_t)$ is the zero polynomial if and only if $\pi
(c(m_1, \dots , m_t) )$ is the zero polynomial. So we have that if
$w_1\wedge \cdots \wedge  w_t$ is the largest exterior monomial
(with respect to $>$) with the property that there is $\alpha \in
GL_n(F)$ with $\ini_{>}(\alpha (f_1) \wedge \cdots \wedge \alpha
(f_t))=w_1\wedge \dots \wedge w_t$, then $\pi(w_1)\wedge \cdots
\wedge  \pi (w_t)$ is the largest exterior monomial (with respect
to $>_{\pi}$) such that there exists $\alpha \in GL_n(F)$ with
$\ini_{>_{\pi}}(\alpha (f_1) \wedge \cdots \wedge \alpha (f_t))=
\pi(w_1)\wedge \cdots \wedge  \pi (w_t)$. Since $(\gin_{>}I)_d$
and $(\gin_{>_{\pi}}I)_d$ are generated by $w_1, \dots ,w_t$ and
$\pi (w_1), \dots ,\pi (w_t)$  (see \cite[4.1.4,
4.1.5]{MR2724673}), respectively and $d$ is arbitrary, the result
follows.
\end{proof}
The subgroup of $GL_i(F)$ consisting of upper triangular matrices
is called the Borel subgroup and we denote it by $\ss_i$.
\begin{remark}
\label{borel} There is a close connection with generic initial
ideals and Borel fixed ideals (which we use in Section
\ref{hilbert}). This connection still exists after the permutation
of variables in the monomial order but one needs to replace the
standard Borel subgroup of upper triangular matrices with the
non-standard Borel subgroup accommodating the permutation.
\end{remark}
\section{Gin of modular Hilbert ideals of cyclic groups of prime order}
\label{hilbert} In this section  $G$ denotes a cyclic group  of
prime order $p$ and we assume that the characteristic of  $F$ is
also $p$. There are exactly $p$ indecomposable $G$-modules $V_1,
\dots , V_{p}$ over $F$ and each indecomposable module $V_n$ is
afforded by a Jordan block of dimension $n$ with 1's on the
diagonal.  If $V$ is a direct sum of indecomposable  modules
$V_{n_1}, \dots , V_{n_k}$, then we write $V=\sum_{1\le j\le
k}V_{n_j}$.
In the sequel we consider the action of the Borel subgroup on
$F[V]$. We always assume that this action is compatible with the
ranking of the variables, i.e.,  the action of a matrix in the
Borel subgroup on the $i$-th variable in the monomial order is
given by the $i$-th column of the matrix, see Remark \ref{borel}.
We compute all generic initial ideals of $H(V)$ for all cases for
which an explicit generating set  for $H(V)$ is known.

\subsection{The monomial cases: $lV_2+mV_3$ and $V_4$.}

For $V=lV_2+mV_3$, we identify  $F[V]$ with   $F[x_i, y_j, z_r, \;
1\le i, j\le l+m,  \; \;  l+1\le r\le l+m]$.  For $1\le i \le l$
$\{x_i, y_i\}$ spans a copy of $V_2^*$ and $\{x_i, y_i, z_i\}$
spans a copy of $V_3^*$ for $l+1\le i \le m+l$. In
\cite[2.6]{MR2193198}  it is shown that $H(lV_2+mV_3)$ is
generated by
$$L=\{x_i, y_i^p \;
  | \; 1\le i\le l \} \cup \{x_i, y_iy_j, z_i^p \; | 1+l\le i, j\le
m+l, \; i\le j \}.$$

\begin{proposition}
\label{iki} There is a monomial order such that $H(lV_2+mV_3)$ is
Borel fixed. Furthermore, there are $(2l+3m)!$ generic initial
ideals of $H(lV_2+mV_3)$. Each of them is generated by $\pi (L)$
for some permutation $\pi$ of the variables in $F[lV_2+mV_3]$.
\end{proposition}
\begin{proof}
Let $>$ be a monomial order with $x_{m+l}> \cdots > x_1>y_{m+l} >
\cdots > y_1> z_{m+l}> \cdots >z_{l+1}$ and let  $\theta \in
\ss_{2l+3m}$. Let $J$ denote the  ideal generated by the subset
$L'=\{x_i \; | \; 1\le i\le l+m \} \cup \{ y_iy_j\; | 1+l\le i,
j\le m+l, \; i\le j \}$ of $L$. Notice that $J$ is a strongly
stable ideal. It follows that $\theta$ sends $J$ into itself. On
the other hand the remaining generators of $H(lV_2+mV_3)$ in
$L\setminus L'$ are pure powers of variables of degree $p$ and
since we are in characteristic $p$,  $\theta$ sends a  $p$-th
power of a variable to a combination of $p$-th powers of
variables. But such a combination is in   $H(lV_2+mV_3)$ as well
since $H(lV_2+mV_3)$ contains all $p$-th powers of variables. It
follows that $H(lV_2+mV_3)$ is Borel fixed and so we have
$\gin_{>}(H(lV_2+mV_3))=H(lV_2+mV_3)$ by \cite[1.8]{MR2150838}.
The final assertion of the proposition follows from Theorem
\ref{permutation}.
\end{proof}
For $V=V_4$, we identify $F[V_4]$ with $F[x_1, x_2, x_3, x_4]$.
From  \cite[3.2]{MR2193198} we get that $H(V_4)$ is generated by
$$M=\{x_1, x_2^2, x_2x_3^{p-3}, x_3^{p-1}, x_4^p\}.$$
\begin{proposition}
\label{dortu} There is a monomial order such that $H(V_4)$ is
Borel fixed. Furthermore, there are $4!$ generic initial ideals of
$H(V_4)$. Each of them is generated by $\pi (M)$ for some
permutation $\pi$ of the variables in $F[V_4]$.
\end{proposition}

\begin{proof}
Let $>$ be a monomial order such that $x_1>x_2>x_3>x_4$ and let
$\theta \in \ss_{4}$. Note that the ideal generated by the subset
$M'=\{x_1, x_2^2, x_2x_3^{p-3}, x_3^{p-1}\}$ is strongly stable
and hence $\theta$ sends this ideal into itself. Furthermore, the
only other generator is a pure $p$-th power of a variable. Since
$\theta$ sends a $p$-th power of a variable to a combination of
$p$-th powers of variables and all
 $p$-th powers of variables are contained in  $H(V_4)$, it follows that
 $H(V_4)$ is Borel fixed. So $\gin_{>}(H(V_4))=H(V_4)$,  by \cite[1.8]{MR2150838}.
The final assertion of the proposition follows from Theorem
\ref{permutation} as in the previous case.
\end{proof}
\subsection{The non-monomial case:
 $V_5$.}

  We identify  $F[V_5]$ with $F[x_1, x_2, x_3, x_4,
 x_5]$. In \cite[4.1]{MR2193198}  it
is shown that

$$T=\{x_1, x_2^2, x_3^2-2x_2x_4-x_2x_3, x_2x_3x_4, x_2x_4^{p-4}, x_3x_4^{p-3}, x_4^{p-1}, x_5^p\}$$
is a generating set for $H(V_5)$ for $p>5$. We denote the
generators of $H(V_5)$ in $T$ with $f_i$ for $1\le i\le 8$ with
$f_1=x_1$ and $f_8=x_5^p$.  Let $\alpha=(\alpha_{ij})$ be an
element in $\ss_5$. Define
$C=\alpha_{33}^{-2}(2\alpha_{23}\alpha_{33}-2\alpha_{22}\alpha_{34}-\alpha_{22}\alpha_{33})$
and $D=\alpha_{33}^{-2}(-2\alpha_{22}\alpha_{44})$. We describe a
generating set for $\alpha (H(V_5))$.
\begin{lemma}
\label{generation} Assume the convention of the previous paragraph
and that $p>5$. Then $\alpha (H(V_5))$ is generated by
$$T':=\{x_1, x_2^2, x_3^2+Cx_2x_3+Dx_2x_4, x_2x_3x_4, x_2x_4^{p-4}, x_3x_4^{p-3}, x_4^{p-1}, x_5^p\}.$$
 \end{lemma}
\begin{proof}
 For $1\le i\le 8$, let $J_i$ denote the ideal
generated by $\alpha (f_1), \dots  , \alpha (f_i)$.
 Since $\alpha$ is a ring homomorphism,  $\alpha (H(V_5))$ is
generated by $\alpha (f_1), \dots  , \alpha (f_8)$ and so we have
$\alpha (H(V_5))=J_8$.  We also denote $x_3^2+Cx_2x_3+Dx_2x_4$
with $f_3'$.

Note that, since $\alpha$ sends $x_1$ to a multiple of $x_1$ and
$x_2$ to a linear combination of $x_1$ and $x_2$ we have that
$J_2=(x_1, x_2^2)$. On the other hand direct computation gives
$\alpha (x_3^2-2x_4x_2-x_2x_3)= (\sum_{1\le i\le
3}\alpha_{i3}x_i)^2-2(\sum_{1\le i\le 2}\alpha_{i2}x_i)(\sum_{1\le
i\le 4}\alpha_{i4}x_i)-(\sum_{1\le i\le
2}\alpha_{i2}x_i)(\sum_{1\le i\le 3}\alpha_{i3}x_i)\equiv
\alpha_{33}^{2} f_3' \mod J_2$. It follows that $J_3=(x_1, x_2^2,
 f_3')$. We finish the proof by showing that $\alpha (f_i)$ is a
 scalar multiple of $f_i$ modulo $J_{i-1}$ for $4\le i\le 8$. This gives   $J_i=(f_i)+J_{i-1}$ for $4\le i\le 8$ and hence $J_8=\alpha (H(V_5))$ is generated by $T'$.

Note that $x_2x_3^2=x_2f_3'-x_2^2(Cx_3+Dx_4)\in J_3$. Therefore,
since  $x_1, x_2^2\in J_3$ as well, we have $\alpha (f_4)=\alpha
(x_2x_3x_4)\equiv \alpha_{22}\alpha_{33}\alpha_{44}x_2x_3x_4$
$\mod J_3$. We also have
$$\alpha (f_5)=\alpha (x_2x_4^{p-4})\equiv \alpha_{22}x_2(\alpha_{24}x_2+\alpha_{34}x_3+\alpha_{44}x_4)^{p-4}\equiv \alpha_{22}\alpha_{44}^{p-4}x_2x_4^{p-4} \mod J_4, $$
where the first equivalence uses $x_1\in J_4$ and the second
equivalence uses that $x_2^2, x_2x_3^2, x_2x_3x_4\in J_4$ and
$p>5$. To compute   $\alpha (f_6)$ we note the identities
$x_3^3=x_3f_3'-Cx_2x_3^2-Dx_2x_3x_4$ and
$x_3^2x_4^{p-4}=x_4^{p-4}f'-Cx_4^{p-5}(x_2x_3x_4)-Dx_4(x_2x_4^{p-4})$
which give that $x_3^3, x_3^2x_4^{p-4}\in J_5$. So
\begin{alignat*}{1} \alpha (x_3x_4^{p-3})&\equiv
(\alpha_{23}x_2+\alpha_{33}x_3)(\alpha_{24}x_2+\alpha_{34}x_3+\alpha_{44}x_4)^{p-3}\\
&\equiv
\alpha_{23}x_2(\alpha_{34}x_3+\alpha_{44}x_4)^{p-3}+\alpha_{33}x_3(\alpha_{24}x_2+\alpha_{34}x_3+\alpha_{44}x_4)^{p-3}\\
&\equiv
\alpha_{33}x_3(\alpha_{24}x_2+\alpha_{34}x_3+\alpha_{44}x_4)^{p-3}\\
&\equiv \alpha_{33}\alpha_{44}^{p-3}x_3x_4^{p-3} \; \mod J_5,
\end{alignat*}
where the first two equivalences use $x_1, x_2^2\in J_5$ and we
have  the third equivalence because $x_2x_3^2,
x_2x_3x_4,x_2x_4^{p-4}\in J_5$. The final equivalence follows
because $x_2^2, x_2x_3x_4, x_3^3, x_3^2x_4^{p-4}\in J_5$ and
$p>5$. For $f_7$ we have
\begin{alignat*}{1} \alpha (x_4^{p-1})&=
(\alpha_{14}x_1+\alpha_{24}x_2+\alpha_{34}x_3+\alpha_{44}x_4)^{p-1} \\
&\equiv
 (\alpha_{34}x_3+\alpha_{44}x_4)^{p-1} \equiv (\alpha_{44}x_4)^{p-1} \; \mod J_6,
\end{alignat*}
where the first equivalence uses that $x_1, x_2^2, x_2x_3x_4,
x_2x_3^2, x_2x_4^{p-4}\in J_6$ and the second one uses
 that $x_3^3,x_3x_4^{p-3}, x_3^2x_4^{p-4}\in J_6$.
 Finally $\alpha (f_8)\equiv \alpha_{55}^p x_5^p \mod J_7$ because $x_i^p\in
 J_7$ for $1\le i\le 4$.
\end{proof}
Note that the sets $T$ and $T'$ differ by one polynomial only. For
the  simplicity of notation we set $f_3=x_3^2+C x_2x_3+D x_2x_4$
and so that $\alpha (H(V_5))$ is generated by $f_i$ for $1\le i\le
8$.

 Fix a term
order $>$ with $x_1> \cdots>x_5$. We set
$$A=\{x_1, x_2^2, x_2x_3, x_3^3, x_3^2x_4, x_2x_4^{p-4}, x_3x_4^{p-3}, x_4^{p-1}, x_5^p\}$$
and
$$B=\{x_1, x_2^2, x_2x_3, x_3^3, x_2x_4^2, x_3^2x_4^{p-5}, x_3x_4^{p-3}, x_4^{p-1}, x_5^p\}.$$

We compute the Gr\"obner basis for $\alpha (H(V_5))$ for a special
class of C.
\begin{lemma}
Let $\alpha=(\alpha_{ij}) \in \ss_5$ such that $C\neq 0$ and
assume that $p>5$. Then $\ini_{>} (\alpha (H(V_5)))$ is generated
by $A$ if $x_3^2>x_2x_4$ and by $B$ otherwise.
\end{lemma}
\begin{proof}
We recall that a Gr\"obner basis can be obtained by reduction of
$S$-polynomials by polynomial division, see
\cite[\S1.7]{MR1287608}.  Before we distinguish  between two
orders we collect  a couple of more elements from $\alpha
(H(V_5))$. The  reduction of the $S$ polynomial  $S(f_2, f_3)$ of
$f_3$ with $f_2$ via $f_2, f_3, f_4$ is $x_3^3$ and the
$S(f_3,f_4)$ is $x_3^2x_4+Dx_2x_4^2$. We denote $x_3^3$ and
$x_3^2x_4+Dx_2x_4^2$ by $f_9, f_{10}$, respectively.

We first consider the case $x_3^2>x_2x_4$. Note then the set $A$
consists of $\ini_{>} ({f_i})$ for $1\le i\le 10$, $i\neq 4$
($f_4=C^{-1}(x_4f_3-f_{10})$ and $\ini_{>} (f_4)$ is divisible by
$\ini_{>} (f_3)$). Therefore it remains to show that this set of
polynomials satisfy the Buchberger criterion. That is,  the
$S$-polynomial of any pair of polynomials $f_i,f_j$ with $1\le i,j
\le 10$ and $i,j\neq 4$ reduces to zero. Since the $S$-polynomial
of two monomials is zero, it suffices to consider the
$S$-polynomials involving either $f_3$ or $f_{10}$. We go through
the pairs  and write the polynomials in the order they appear in
the polynomial division: $S(f_2, f_3)$ reduces to zero via $f_2,
f_3, f_9, f_{10}$. Both $S$-polynomials $S(f_3, f_5)$ and $S(f_3,
f_6)$ reduce to zero via $f_{10}$. The $S$-polynomial $S(f_3,
f_9)$ reduces to zero via $f_9$, $f_3$ and $f_{10}$ and the
$S$-polynomial $S(f_3, f_{10})$ reduces to zero via $f_2, f_3,
f_9,f_{10}$. The $S$-polynomials $S(f_5,f_{10})$, $S(f_6,f_{10})$,
$S(f_7,f_{10})$ reduce to zero at one step each via $f_5$.
Finally, $S(f_9, f_{10})$ reduces to zero via $f_3$ and $f_{10}$.
We have considered all  pairs  whose initial terms are not
relatively prime. So the proof for this case is complete because
the $S$-polynomial of two polynomials that have relatively prime
initial terms reduces to zero.

Now we consider the case  $x_2x_4>x_3^2$. Define
$f_{11}=x_4^{p-6}f_{10}-Df_5=x_3^2x_4^{p-5}$. Since $D$ is a
non-zero scalar, from this equality we have that $f_i$ for $1\le
i\le 11$ and $i\neq 4,5$ also generate $\alpha (H(V_5))$. Since
the set $B$ consists of $\ini_{>} (f_i)$ for $1\le i\le 11$ and
$i\neq 4,5$, it remains to show that  this set satisfies the
Buchberger criterion. As in the previous case we just need to
check pairs whose initial terms are not relatively prime and one
of the polynomials in the pair is not monomial. We note that the
reductions of the $S$-polynomials involving $f_3$ in the previous
case carry over to this case except $S(f_3, f_{10})$. To see this,
first note that the missing initial monomial $\ini_> (f_5)$ is not
used in these reductions. Also the initial monomials of  $f_i$ for
$1\le i\ \le 10$ do not change with the change of the order except
$f_{10}$. But in these reductions $f_{10}$ is not used in the
polynomial division except at the last step. So the last non-zero
remainder is a multiple of $f_{10}$ giving that this last
polynomial division is also a reduction in the second order as
well. Finally, $S(f_3,f_{11})$ reduces to zero via $f_9, f_3,
f_{10}$. We finish the proof by checking the pairs involving
$f_{10}$. We have that $S(f_3, f_{10})$ reduces to zero via $f_9,
f_{10}$. The $S$-polynomial $S(f_2,f_{10})$ reduces to zero via
$f_3, f_{10}$ and the $S$-polynomial $S(f_7,f_{10})$  reduces to
zero via $f_6$. Both $S(f_{6},f_{10})$ and $S(f_{10},f_{11})$
reduce to zero via $f_9$.
\end{proof}
\begin{theorem}
Assume that $p>5$. There are $2(5!)$ generic initial ideals of
$H(V_5)$. Each of them is generated by $\pi(A)$ or $\pi(B)$, where
$\pi$ is a  permutation of the variables in $F[V_5]$.
\end{theorem}
\begin{proof}
Let $\alpha=(\alpha_{ij}) \in \ss_5$. If $C= 0$, then from Lemma
\ref{generation} we get that the monic generators of $\ini_{>}
(\alpha (H(V_5)))$ of degree at most two are $x_1, x_2^2, x_3^2$
if $x_3^2>x_2x_4$ and are  $x_1, x_2^2, x_2x_4$ otherwise. Note
that $x_2x_3\notin \ini_{>} (\alpha (H(V_5)))$ in both cases and
so $\ini_{>} (\alpha (H(V_5)))$ fails to be Borel fixed because
either $x_3^2$ or $x_2x_4$ lies in $\ini_{>} (\alpha (H(V_5)))$.
It follows that $\gin_{>}(H(V_5))\neq \ini_{>} (\alpha (H(V_5)))$
if $C=0$ because generic initial ideals are always Borel fixed ,
see for instance \cite[\S 15]{MR1322960}. On the other hand, by
the previous lemma all other members of $\ss_5$ generate the same
initial ideal. But it is a standard fact that at least one element
in $\ss_5$ generates the generic initial ideal
(\cite[15.18]{MR1322960}) and so $\gin_{>} (H(V_5))=\ini_{>}
(\alpha (H(V_5)))$ for $\alpha\in \ss_5$ satisfying $C\neq 0$.
This initial ideal is generated by $A$ or $B$ depending how $>$
compares $x_2x_4$ and $x_3^2$. The final assertion of the theorem
follows from Theorem \ref{permutation}.
\end{proof}
\begin{remark}
\label{v5}
\begin{enumerate}
\item The computation of $\gin_>(H(V_5))$ for $p=5$ is very
similar and there are only small changes in the generating set. We
include the details in the version we submit to  arXiv preprint
server. \item The identity matrix satisfies that $C\neq 0$.
Therefore from the proof of the theorem we get that
$\gin_>(H(V_5))=\ini_>(H(V_5))$ for all monomial orders with $x_1>
\cdots >x_5$.
\end{enumerate}
\end{remark}
\section{A speculation}
\label{klein}
 We first note that Hilbert ideals  of the Klein
 four group over charateristic two is Borel fixed with the right ordering of the variables in the monomial order.
\begin{proposition} Let $G$ denote the Klein four
group and let $W$ be a $G$-module over characteristic two not
containing the regular representation as a summand. Then there is
choice for a basis for $W^*$ and a ranking of variables in $F[W]$
such that $H(W)$ is Borel fixed for all monomial orders compatible
with this ranking. In particular $\gin_{>}(H(W))=H(W)$ for all
such orders $>$.
\end{proposition}
\begin{proof}
Generators of $H(W)$ have been studied in \cite{MR3458394} and
\cite{MR3859223}. From  \cite[Proposition 16, 17]{MR3859223} we
see that there is a basis for $W^*$ such that $H(W)$ is generated
by first, second and forth powers of these basis elements (recall
that $F[W]$ is a polynomial ring in these basis elements). Say
$F[W]=F[x_1, \dots , x_n]$ and $H(W)$ is generated by $x_i^{a_i}$,
where $a_i\in \{1, 2, 4\}$ for $1\le i\le n$. Let $>$ be a
monomial order such that $x_i>x_j$ whenever $a_i < a_j$. Since we
are in characteristic two and $a_i$ is a 2-th power, a member of
the non-standard Borel subgroup (see Remark \ref{borel}) sends
$x_i^{a_i}$  to some combination of $a_i$-th powers of variables
of higher or equal rank. By the choice of $>$, this combination is
also in $H(W)$, so $H(W)$ is Borel fixed. Therefore,
$\gin_{>}(H(W))=H(W)$ by \cite[1.8]{MR2150838}.
\end{proof}
Reviewing the data we have collected so far, the monomial Hilbert
ideals of the cyclic group of prime order are Borel fixed
(assuming $x_1> x_2> \cdots  $) and so the gin of the Hilbert
ideal is equal to the Hilbert ideal itself, see Propositions
\ref{iki} and \ref{dortu}. The non-monomial case, with the same
ordering,  still satisfies $\gin_>(H(V_5))=\ini_>(H(V_5))$, see
Remark \ref{v5}. Together with the  previous proposition on the
Klein four group we feel that there is enough ground for the
following conjecture.

\begin{conjecture}
\label{conj} Let $G$ be a $p$-group and $V$ be a $G$-module over a
field $F$ of characteristic $p$. Then there is a choice of a basis
for $V^*$ and a  monomial order $>$ on the monomials in $F[V]$
such that
$$\gin_{>}(H(V))=\ini_{>} (H(V)).$$
\end{conjecture}


\bibliographystyle{plain}
\bibliography{OurBib}
\end{document}